\documentclass[11pt,oneside,reqno,a4paper,openany]{amsart}

\usepackage{amscd}
\usepackage{amssymb}
\usepackage{amstext}
\usepackage{amsthm}
\usepackage{mathrsfs}
\usepackage{latexsym}
\usepackage{xcolor}
\numberwithin{equation}{section}
\usepackage{hyperref}

\newcommand{\cc}{\mathbb{C}}
\newcommand{\N}{\mathbb{N}}
\newcommand{\R}{\mathbb{R}}
\renewcommand{\L}{\mathcal{L}}
\newcommand{\Bn}{\mathbb B_n}
\newcommand{\Sn}{\mathbb S_n}
\newcommand\set[1]{\left\{#1\right\}}
\providecommand{\abs}[1]{\lvert#1\rvert}

\newcommand{\lambdaw}{\lambda_\rho}
\newcommand{\vhw}[1]{\langle #1\rangle_\rho}

\def\A2w{A^2_\rho}

\providecommand{\normw}[1]{\lVert#1\rVert_\rho}
\def\id{\operatorname{Id}}
\newcommand{\less}{\lesssim}
\newcommand{\gess}{\gtrsim}
\newcommand{\asym}{\asymp}
\newcommand{\vh}[1]{\langle #1\rangle}
\providecommand{\norm}[1]{\lVert#1\rVert}

\newtheorem{Thm}{Theorem}[section]
\newtheorem{theorem}[Thm]{Theorem}
\newtheorem{lemma}[Thm]{Lemma}
\newtheorem{proposition}[Thm]{Proposition}
\newtheorem{corollary}[Thm]{Corollary}

\theoremstyle{definition}
\newtheorem{remark}[Thm]{Remark}

\begin{document}
\sloppy
\title[Carleson measures and Toeplitz operators]
{Carleson measures and Toeplitz operators on small Bergman spaces on the ball}

\author{Van An Le}
\address{Aix--Marseille University, CNRS, Centrale Marseille, I2M, Marseille, France}
\address{University of Quynhon, Department of Mathematics, 170 An Duong Vuong, Quy Nhon, Vietnam}
\email{vanandkkh@gmail.com}


\begin{abstract}
We study the Carleson measures and the Toeplitz operators on the class of so-called small weighted Bergman spaces, introduced recently by Seip. A characterization of Carleson measures is obtained which extends Seip's results from the unit disc of $\cc$ to the unit ball of $\cc^n$. We use this characterization to give necessary and sufficient conditions for the boundedness and compactness of Toeplitz operators. Finally, we study the Schatten $p$ classes membership of Toeplitz operators for $1<p<\infty$.
\end{abstract}

\maketitle

\section{Introduction}

Let $\cc^n$ denote the $n-$dimensional complex Euclidean space, $\Bn=\{z\in \cc^n:|z|<1\}$ be the unit ball and $\Sn=\set{z\in \cc^n: \abs z=1}$ be the unit sphere in $\cc^n$.
Denote by $H(\Bn)$ the space of all holomorphic functions on the
unit ball $\Bn$.
Let $dv$ be the normalized volume measure on $\Bn$.
The normalized surface measure on $\Sn$ will be denoted by $d\sigma$.

Let $\rho$ be a positive continuous and integrable function on $[0,1)$. We  extend it to $\Bn$ by $\rho(z)=\rho(\abs z)$, and call such $\rho$ a weight function. The weighted Bergman space $\A2w$ is the space of functions $f \text{ in } H(\Bn)$ such that 
 $$ \normw{f}^2 =\int_{\Bn}\abs{f(z)}^2\rho( z)dv(z)<\infty. $$
Note that $\A2w$ is a closed subspace of $L^2(\Bn, \rho dv)$ and hence is a Hilbert space  endowed with the inner product 
 $$\vhw{f,g}=\int_{\Bn}f(z)\overline{g(z)}\rho(z)dv(z),\qquad f,g \in \A2w.$$
 When $\rho(r)=(1-r^2)^\alpha, \alpha >-1$, we obtain the standard Bergman spaces $A^{2}_\alpha$.

We impose a normalization condition on $\rho$:
$$ \int_0^1 x^{2n-1} \rho(x)dx=1. $$

Consider the points $r_k\in [0,1)$ determined by the relation 
$$\int_{r_k}^1 \rho(x)dx=2^{-k}.$$

Denote by $S$ the class of weights $\rho$ such that
\begin{equation} \label{1}
\inf_{k} \dfrac{1-r_k}{1-r_{k+1}}>1.
\end{equation}

Since the function $r\mapsto \int_{\Sn}\abs{f(r\xi)}^2d\sigma(\xi)$ is non-decreasing, we also have the equivalent norm
\begin{equation}\label{*}
\normw{f}^2\asymp\sum_{k=1}^\infty 2^{-k}\int_{\Sn}\abs{f(r_k\xi)}^2d\sigma(\xi), \qquad f\in \A2w.
\end{equation}

The class $S$ was introduced by  Kristian Seip in  \cite{Seip2013}.
It is easy to see that the functions
$$\rho(x)=(1-x)^{-\beta},\qquad 0<\beta<1, $$
and
$$\rho(x)=(1-x)^{-1}\left (\log \dfrac{1}{1-x}\right )^{-\alpha}, \qquad 1 <\alpha <\infty,$$
belong to $S$.

In this paper we prove a characterization of Carleson measure for weighted Bergman spaces $\A2w$, with $\rho\in S$. This result is then used to study spectral properties of Toeplitz operators on these spaces.

Let $\mu$ be a finite positive Borel measure on $\Bn$. We say that $\mu$ is a Carleson measure for a Hilbert space $X$ of analytic functions in $\Bn$ if there exists a positive constant $C$ such that 
$$ \int_{\Bn}\abs{f(z)}^2d\mu(z)\le C\norm{f}_X^2, \qquad f\in X. $$
It is clear that  $\mu$ is a Carleson measure for $\A2w$ if and only if $\A2w\subset L^{2}(\Bn, d\mu)$ and the identity operator $\id: \A2w\to L^{2}(\Bn, d\mu)$ is bounded. The Carleson constant  of $\mu$, denoted by $\mathcal{C}_\mu(\A2w)$, is the norm of this identity operator $\id$.
Suppose that $\mu$ is a Carleson measure for $\A2w$.  We say that $\mu$ is a vanishing Carleson measure for $\A2w$ if the above identity operator $\id$ is compact. That is,
$$ \lim_{k\to\infty}\int_{\Bn}\abs{f_k(z)}^2d\mu(z)=0$$ 
whenever $\set{f_k}$ is a bounded sequence in $\A2w$ which converges to $0$ uniformly on compact subsets of $\Bn$.

The concept of a Carleson measure was first introduced by L. Carleson \cite{Carleson1, Carleson2} in order to study interpolating sequences and the corona problem on the algebra $H^\infty$ of all bounded analytic functions on the unit disk. It quickly became a powerful tool for the study of function spaces and operators acting on them. The Carleson measures on Bergman spaces were studied by Hastings \cite{Hastings}, and later on by Luecking \cite{Luecking1}, and many others. Recently, Pau and Zhao \cite{pau2015} gave a characterization for Carleson measures and vanishing Carleson measures on the unit ball by using the products of functions in weighted Bergman spaces. In  \cite{pelaez}, Pel\'aez and R\"atty\"a gave a description of Carleson measures for $\A2w$ on unit disk when $\rho$ is such that $\dfrac{1}{(1-r)\rho(r)}\int_r^{1}\rho(t)dt$ is either equivalent to $1$ or tends to $\infty$, and in \cite{pelaez2015} they then got a criterion for $\A2w$ on unit disk when $\rho \in \widehat{\mathcal{D}}$, which means $\int_{r}^1 \rho(s)ds \lesssim \int_{\frac{r+1}{2}}^1\rho(s)ds$. 

In \cite{Seip2013}, Seip gave a characterization of Carleson measures for $\A2w$ with $\rho \in S$ in the case $n=1$. One of our main results, Theorem \ref{Carleson}, extends this result to the case $n>1$.

Given a function $\varphi\in L^\infty(\Bn)$, the Toeplitz operator $T_\varphi$ on $\A2w$ with symbol $\varphi$ is defined by
$$
T_\varphi f = P(\varphi f), \qquad f\in \A2w,			$$
where $P: L^2(\Bn,\rho dv) \to \A2w$ is the orthogonal projection onto $\A2w$. Using the integral representation of $P$, we can write $T_\varphi$ as 
$$T_\varphi f(z)=\int_{\Bn}K_\rho(z,w)f(w)\varphi(w)\rho(w)dv(w),\qquad z\in \Bn,$$
where $K_\rho(z,w)$ is the reproducing kernel for $\A2w$. The Toeplitz operators can also be defined for unbounded symbols or for finite measures on $\Bn$. In fact, given a finite positive Borel measure $\mu$ on $\Bn$, the Toeplitz operator $T_\mu: \A2w \to \A2w$ is defined as follows:
$$ T_\mu f(z)=\int_{\Bn} K_\rho(z,w)f(w)d\mu(w), \qquad z\in \Bn. $$
Note that 
$$\vhw{T_\mu f,g}=\int_{\Bn} f(z)\overline{g(z)}d\mu(z), \qquad f,g\in \A2w.$$
The Toeplitz operators acting on various spaces of holomorphic functions have been extensively studied by many authors, and the  theory is especially well understood in the case of Hardy spaces or standard Bergman spaces (see \cite{zhuball}, \cite{zhudisk} and the references therein). Luecking \cite{Luecking2} was the first to study Toeplitz operators on Bergman spaces with measures as symbols, and some interesting results about Toeplitz acting on large Bergman spaces were obtained by Lin and Rochberg \cite{LinRoc96}. In this paper, we will study the boundedness and compactness of $T_\mu$ on $\A2w$, with $\rho \in S$.

Next we study when our Toeplitz operators belong to the Schatten class. We refer to \cite[Chapter 1]{zhudisk} for a brief account on the Schatten classes. A description for the standard Bergman spaces on the unit disk was given (see \cite[Chapter 7]{zhudisk}), and a  description for the case of large Bergman spaces on the disk was obtained in 2015 by H. Arroussi, I. Park, and J. Pau \cite{Arroussi2015}.  In 2016, Pel\'aez and R\"atty\"a \cite{pelaez2016} gave an interesting characterization for the case of small Bergman spaces on unit disk, where the weight $\rho \in \widehat{\mathcal{D}}$. Note that $S \subsetneqq \widehat{\mathcal{D}}$, but $\set{\A2w:\rho\in S}=\set{\A2w:\rho\in \widehat{\mathcal{D}}}$.

We introduce a subclass $S^*$ of weights in $S$ determined by the condition that $\rho^*(r)\lesssim \rho(r)$ for $r\in (0,1)$, where
$$\rho^*(r)=\dfrac{1}{1-r}\int_r^1 \rho(t)dt.$$

For example, the weights 
$$\rho(x)= (1-x)^{-\beta}\left (\log \frac{1}{1-x}\right )^{\alpha}, \qquad 0<\beta<1, \alpha\in \R$$
belong to $S^*$, but the weights
$$\rho(x)=(1-x)^{-1}\left(\log \frac{1}{1-x} \right)^{\alpha}, \qquad \alpha<-1,$$
$$\rho(x)=(1-x)^{-1}\left(\log \frac{1}{1-x} \right)^{-1}\left(\log\log \frac{1}{1-x} \right)^{\alpha}, \qquad \alpha<-1,$$
do not belong to $S^{*}$.

For weights $\rho$ in $S^*$, we obtain a characterization of the symbols of the Toeplitz operators in the Schatten classes $\mathcal S_p$. In \cite{pelaez2018}, Pel\'aez, R\"atty\"a and Sierra gave a characterization for the case of dimension $n=1$ when the weight is regular, that is $\rho^*(r)\asymp \rho(r)$. As an easy observation, our result is equivalent to their result when $n=1$. We point out that our approach is completely different from that of \cite{pelaez2018}, which does not seem to work in higher dimensions.
 On the other hand, for regular  weights $\rho$ in $S\setminus S^*$, this characterization fails. A counterexample was given in \cite{pelaez2018}.

The paper is organized as follows: The main results are stated in Section \ref{sec2} and their proofs are given in Section \ref{sec3}--\ref{sec5}. 

\section{Main results}\label{sec2}

Throughout this text, we use the following notation. For every nonnegative integer $k$, set
$$ \Omega_k=\set{z\in \Bn: r_k\le \abs z<r_{k+1}}, $$
and let $\mu_k$ be the measure defined by $\mu_k=\chi_{\Omega_k}\mu$ whenever a nonnegative Borel measure $\mu$ on $\Bn$ is given. The notation $U(z)\lesssim V(z) $ (or equivalently $V(z)\gess U(z)$) means that there is a positive constant $C$ such that $U(z)\le CV(z)$ holds for all $z$ in the set in question, which may be a space of functions or a set of numbers. If both $U(z)\less V(z)$ and $V(z)\less U(z)$, then we write $U(z)\asym V(z)$.

Our results are following:
\begin{theorem}\label{Carleson}
Let $\rho\in S$, and let $\mu$ be a finite positive Borel measure on $\Bn$. Then
\begin{itemize}
\item[(i)] $\mu$ is a Carleson measure for $\A2w$ if and only if each $\mu_k$ is a Carleson measure for the Hardy space $H^2$ with Carleson constant $\mathcal C_{\mu_k}(H^2)\lesssim 2^{-k}, k \ge 0$.
\item[(ii)] $\mu$ is a vanishing Carleson measure for $\A2w$ if and only if 
\begin{equation*}
\lim_{k\to \infty} 2^k\mathcal C_{\mu_k}(H^2)=0.
\end{equation*}
\end{itemize}
\end{theorem}

Theorem \ref{Carleson} (i) for the case $n=1$ was obtained by Seip in \cite{Seip2013}.

\begin{theorem}\label{Toeplitz} Let $\rho\in S$, and let $\mu$ be a finite positive Borel measure on $\Bn$. Then
\begin{itemize}
\item[(i)] The Toeplitz operator $T_\mu $ is bounded on $\A2w$ if and only if $\mu$ is a Carleson measure for $\A2w$.
\item[(ii)] The Toeplitz operator $T_\mu$ is compact on $\A2w$ if and only if $\mu$ is a vanishing Carleson measure for $\A2w$.
\end{itemize}
\end{theorem}

Given $z\in \Bn$ and $0<\alpha<1$, we consider the Bergman metric ball
$$E(z,\alpha)=\set{w\in \Bn: \beta(z,w)<\alpha},$$
where $\beta(z,w)$ is the Bergman metric given by 
$$
\beta(z,w)=\dfrac12\log\dfrac{1+\abs{\varphi_z(w)}}{1-\abs{\varphi_z(w)}}, \qquad z, w \in \Bn.
$$
Here, $\varphi_z$ is the M\"obius transformation on $\Bn$ that interchanges $0$ and $z$.

 We know that $E(0,\alpha)$ is actually a Euclidean ball of radius $R=\tanh\alpha$, centered at the origin, and
$$ E(z,\alpha)=\varphi_z\bigl(E(0,\alpha)\bigr). $$
Moreover, for fixed $\alpha$, $v\bigl(E(z,\alpha)\bigr)\asymp (1-\abs z)^{n+1}$. See \cite[Chapter 1]{zhuball} for more details. 

For a measure $\mu$ on $\Bn$ and $\alpha>0$, we define the function $\widehat{\mu}_\alpha$ by
$$\widehat{\mu}_\alpha(z)=\dfrac{2^k\mu\bigl(E(z,\alpha)\bigr)}{(1-\abs z)^n},\qquad z\in \Omega_k.$$

Denote by $\widetilde{T_\mu}$ the Berezin transform  of $T_\mu$, and set 
$$d\lambdaw(z)=\dfrac{2^k \rho(z) dv(z)}{(1-\abs z)^n}, \qquad z\in \Omega_k.$$
\begin{theorem}\label{Schatten}
Let $\rho$ be in $S^*$, $\mu$ be a finite positive Borel measure and $1<p<\infty$. The following conditions are equivalent:
\begin{itemize}
\item[(a)] The Toeplitz operator $T_\mu$ is in the Schatten class $\mathcal S_p$.

\item[(b)] The function $\widetilde{T_\mu}$ is in $L^{p}(\Bn, d\lambdaw).$

\item[(c)] The function $\widehat{\mu}_\alpha$ is in $L^{p}(\Bn, d\lambdaw)$ for sufficiently small  $\alpha>0$.
\end{itemize}
\end{theorem}

\section{Proof of Theorem \ref{Carleson}}\label{sec3}

Given $a\in \Bn\setminus \set{0}$ and $r>0$. Let $\delta(a)=\sqrt{2(1-\abs a)}$. Define $Q(a,r)\subset \Bn$ and $O(a,r)\subset \Sn$ as follows:
$$
Q(a,r)=\{z\in \mathbb B_n: \sqrt{|1-\langle a/|a|, z\rangle|} < r\},$$
$$
O(a,r)= \{\zeta\in \mathbb S_n: \sqrt{|1-\langle a/|a|, \zeta\rangle|} < r\}.$$
For simplicity of notation, we write $Q_a$ instead of $Q\bigl(a,\delta(a)\bigr)$, $O_a$ instead of $O\bigl(a,\delta(a)\bigr)$.

We recall a well known characterization of Carleson measures for the Hardy space (see \cite{zhuball}): \textit{A positive Borel measure $\mu$ on $\Bn$ is a Carleson measure for $H^2$ if and only if $\mu(Q_a) \lesssim (1-\abs a)^{n}$ for all $a\in \Bn\setminus \set{0}$. Furthermore, $\mathcal C_\mu(H^2)\asymp \sup _{a\in \Bn\setminus \set{0}}\mu(Q_a)(1-\abs a)^{-n}.$}

We use the following covering lemma from \cite[Lemma 4.7]{zhuball}.

\begin{lemma}\label{lma}
Suppose $N$ is a natural number, $a_l \in \Bn\setminus \set{0}, 1\le l\le N$, $$E=\bigcup_{l=1}^{N} O_{a_l}.$$
There exists a subsequence $\{l_i\}, 1\le i\le M$, such that 
\begin{itemize}
\item[(a)] $O_{a_{l_i}}, 1\le i\le M$, are disjoint.
\item[(b)] $O\big(a_{l_i}, 3\delta(a_{l_i})\big), 1\le i\le M$, cover $E$.
\end{itemize}
\end{lemma}

\begin{lemma}\label{lem*1}
Let $\mu$ be a finite positive measure on $\Bn$. Then $\mu_k$ is a Carleson measure for $H^2$ if and only if $\mu_k(Q_a)\less (1-\abs a)^n$ for all $a\in \Omega_k$. Furthermore, $\mathcal C_{\mu_k}(H^2)\asymp \sup _{a\in \Omega_k}(1-\abs a)^{-n}\mu_k(Q_a).$
\end{lemma}

\begin{proof}
Let $a\in \Bn\setminus \set{0}$. Then $a\in \Omega_l$ for some $l\ge 1$. If $l > k$, then $\mu_k(Q_a)=0$ and there is nothing to prove. 
When $a\in \Omega_l,\, l\le k$, we can cover $Q_a \setminus r_{k}\mathbb B_n$ by a finite family  $\Lambda$  of  $Q_{a_l}$ with $a_l\in \Omega_{k-1}$. Denote by $\Lambda_0$ the sub-family of $\Lambda$  obtained by Lemma \ref{lma}. Then
$$\mu_k(Q_a) = \mu_k(Q_a \setminus r_{k}\mathbb B_n) \le \sum_{l\in \Lambda_0}\mu_k\Big(Q\big({a_l}, 3\delta(a_l)\big)\Big). $$
Since $a_l\in \Omega_{k-1}$, we have $\mu_k\Big(Q\big({a_l}, 3\delta(a_l)\big)\Big) \lesssim (1-|a_l|)^n\asymp \sigma(O_{a_l})$. Hence
$$\mu_k(Q_a)\lesssim \sum_{l\in\Lambda_0}\sigma(O_{a_l})=\sigma\Bigl(\bigcup_{l\in\Lambda_0}O_{a_l}\Bigr).$$
Finally,
$$\sigma \Bigl(\bigcup_{l\in\Lambda_0}O_{a_l}\Bigr) \lesssim \sigma(O_a) \asymp (1-|a|)^n.$$
Therefore $\mu_k(Q_a) \lesssim (1-|a|)^n.$ This completes the proof.
\end{proof}

\subsection{Proof of Part (i)}

$(\Longleftarrow)$ Since $\mu_k$ are Carleson measures for $H^2$ with Carleson constants $\lesssim 2^{-k}$, the same holds for $H^2$ on the smaller ball $r_{k+2}\Bn$. This means that
$$ \int_{\Omega_k}\abs{f(z)}^2d\mu(z)\lesssim 2^{-k}\int_{\Sn}\abs{f(r_{k+2}\xi)}^2d\sigma(\xi) $$
for an arbitrary function $f$ in $\A2w$ and for all $k.$ Summing this estimate over all $k\ge 1$ we get
$$ \int_{\Bn} \abs{f(z)}^2d\mu(z)\lesssim \sum_{k=1}^\infty 2^{-k}\int_{\Sn}\abs{f(r_{k+2}\xi)}^2d\sigma(\xi)\asymp \normw{f}^2.$$

$(\Longrightarrow)$ We just need to check that $\mu_k(Q_a)\less 2^{-k}(1-\abs a)^n$ when $a$ is in $\Omega_k, k\ge 0$. We use the test function 
\begin{equation}\label{2}
f_a(z)=(1-\vh{a,z})^{-\gamma}
\end{equation}
 with large $\gamma$. By \eqref{*}, we have
\begin{align*}
\normw{f_a}^2&\asymp\sum_{j=1}^\infty 2^{-j}\int_{\Sn}\dfrac{1}{\abs{1-\vh{a,r_j\xi}}^{2\gamma}}d\sigma(\xi)\\ 
& \asymp \sum_{j=1}^\infty \dfrac{2^{-j}}{ (1-r_j\abs a)^{2\gamma-n}}.
\end{align*}
Since $a\in \Omega_k$, by Property \eqref{1} we obtain
\begin{equation}\label{3}
\normw{f_a}^2\asymp 2^{-k} (1-\abs a)^{-2\gamma+n}.
\end{equation}

On the other hand, for every $z$ in $Q_a$ we have
\begin{align*}
\abs{1-\vh{a,z}}&= \abs{(1-\abs a)+\abs a(1-\vh{a/\abs a,z})}\\ 
& \le (1-\abs a)+\abs a \abs{1- \vh{a/\abs a,z}}\\
&< (1-\abs a) + 2\abs a (1-\abs a) \\
&\le 3 (1-\abs a).
\end{align*}
Hence, 
\begin{equation}\label{4}
\abs{f_a(z)} \gtrsim (1-\abs a)^{-\gamma}, \qquad z\in Q_a.
\end{equation}
Thus, 
$$ \int_{\Bn}\abs{f_a(z)}^2d\mu(z)\gtrsim  (1-\abs a)^{-2\gamma}\mu(Q_a\cap \Omega_k).$$
Since $\mu$ is a Carleson measure for $\A2w$, we get
$$ \mu(Q_a\cap \Omega_k)\lesssim 2^{-k} (1-\abs a)^{n}. $$
This implies that $\mu_k$ is a Carleson measure for Hardy space $H^2$ with Carleson constant $\mathcal C_{\mu_k}(H^2)\lesssim 2^{-k}$. 
\hfill $\Box$

\subsection{Proof of Part (ii)}

Suppose that $\mu$ is a vanishing Carleson measure for $\A2w$. 
Given $a$ in $\Omega_k$, consider the function $f_a$ defined by \eqref{2}. By \eqref{3}, $ \normw{f_a}^2  \asymp 2^{-k} (1-\abs a)^{-2\gamma+n}.$ Set 
\begin{equation}\label{5}
h_a(z)=\dfrac{(1-\vh{a,z})^{-\gamma}}{ 2^{-k/2} (1-\abs a)^{-\gamma+n/2}}.
\end{equation}
Then $\normw{h_a}^2  \asymp 1$ and by \eqref{4},
$$\abs{h_a(z)}^2\gess \dfrac{2^k}{(1-\abs a)^n}, \qquad z\in Q_a.$$
Since $\mu$ is a vanishing Carleson measure for $\A2w$ and $h_a$ tends to $0$ uniformly on compact subsets of the unit ball as $\abs a\to 1$, we have
$$
\lim _{\abs a \to 1}\int_{\Bn}\abs{h_a(z)}^2d\mu(z)=0.
$$
Thus, $\displaystyle\sup_{a\in \Omega_k}\dfrac{2^k\mu_k(Q_a\cap \Omega_k)}{(1-\abs a)^n}\to 0$ as $k\to \infty$. Hence, $\displaystyle\lim_{k\to \infty} 2^k\mathcal C_{\mu_k}(H^2)=0.$

Conversely, let $\mu^r=\mu|_{\Bn\setminus \overline{r\Bn}}$, where $r\Bn=\set{z\in \Bn:\abs z < r}$. Then part $(i)$ of Theorem \ref{Carleson} implies that 
$$\int_{\Bn}\abs{h(z)}^2d\mu^r(z)\le C_r\normw{h}^2, \qquad h\in \A2w,$$
where 
\begin{equation}\label{6}
C_r=\sup_{k: r_k>r}\mathcal C_{\mu_k}(H^2),\qquad \text{ and } \qquad
\lim_{r\to 1} C_{r}=0.
\end{equation}
Let $\set{f_k}$ be a bounded sequence in $\A2w$ converging uniformly to $0$ on compact subsets of $\Bn$. Let $\varepsilon >0$. By \eqref{6}, there exists $r_0\in (0,1)$ such that $C_{r}<\varepsilon$ for all $r\ge r_0$. Moreover, by the uniform convergence on compact subsets, we may choose $k_0\in \N$ such that $\abs{f_k(z)}^2<\varepsilon$ for all $k\ge k_0$ and $z\in \overline{r_0\Bn}$. It follows that
\begin{align*}
\int_{\Bn}\abs{f_k(z)}^2d\mu(z)&=\int_{\overline{r_0\Bn}}\abs{f_k(z)}^2d\mu(z) +\int_{\Bn\setminus \overline{r_0\Bn}}\abs{f_k(z)}^2d\mu(z)\\
		&<\varepsilon\mu(\overline{r_0\Bn})+\int_{\Bn}\abs{f_k(z)}^2d\mu^{r_0}(z)\\
		&\le \varepsilon\mu(\overline{r_0\Bn}) + C_{r_0}\normw{f_k}^2\\
		&\le \varepsilon C, \qquad k\ge k_0,
\end{align*}
for some positive constant $C$. Hence, $\mu$ is a vanishing Carleson measure for $\A2w$.
\hfill $\Box$

\section{Proof of Theorem \ref{Toeplitz}}\label{sec4}
\subsection{Proof of Part (i)}

$(\Longrightarrow)$ Given $a$ in $\Omega_k$, we define $h_a$ by \eqref{5}. 
Then 
$$\normw{h_a}^2  \asymp 1 \text{ and }\abs{h_a(z)}^2\gess 2^k(1-\abs a)^{-n}, \qquad z\in Q_a.$$ Consider the function
\begin{equation}\label{7}
\widetilde{ T_\mu}(a)=\vhw{T_\mu h_a,h_a}=\int_{\Bn}\abs{h_a}^2d\mu(z).
\end{equation}
Since $T_\mu$ is bounded, $A:= \sup_{a\in \Bn}\widetilde{ T_\mu}(a) <\infty.$ Then
\begin{align}
A&\ge \int_{\Bn}\abs{h_a(z)}^2d\mu(z) \ge \int_{\Bn}\abs{h_a(z)}^2d\mu_k(z)\notag\\
& \ge \int_{Q_a}\abs{h_a(z)}^2d\mu_k(z)\gess 2^k(1-\abs a)^{-n} \mu_k(Q_a). \label{8}
\end{align}
Hence, $\mu_k(Q_a)\lesssim 2^{-k}(1-\abs a)^{n}$ for every $a\in \Omega_k$. By Theorem \ref{Carleson} and Lemma \ref{lem*1}, $\mu$ is Carleson measure for $\A2w$.

$(\Longleftarrow)$ For every $f,g \in \A2w$ we have
$$\vhw{T_\mu f,g}=\int_{\Bn}f(z) \overline{g(z)}d\mu(z).$$
Then by Cauchy--Schwarz inequality, we get
\begin{align*}
\abs{\vhw{T_\mu f,g}} &\le \int_{\Bn}\abs{f(z)}\abs{g(z)}d\mu(z)\\
					&\le \left(\int_{\Bn}\abs{f(z)}^2d\mu(z)\right)^{\frac12} \left(\int_{\Bn}\abs{g(z)}^2d\mu(z)\right)^{\frac12}.
\end{align*}
Since $\mu$ is a Carleson measure for $\A2w$, there exists a positive constant $C$ such that
\begin{align*}
\int_{\Bn}\abs{f(z)}^2d\mu(z)&\le C \normw{f}^2,\\
\intertext{and}
\int_{\Bn}\abs{g(z)}^2d\mu(z)&\le C \normw{g}^2.
\end{align*}
Hence, 
$$\abs{\vhw{T_\mu f,g}} \le C \normw{f} \normw{g} \qquad \text{ for all } f, g\in \A2w.$$
Thus, $T_\mu$ is bounded on $\A2w$.
\hfill $\Box$

\subsection{Proof of part (ii)}
We need the following auxiliary results.
\begin{proposition}\label{md4}
Suppose that $f\in \A2w$ with $\rho\in S$. Then
\begin{equation}\label{9}
\abs{f(z)}^2\le \dfrac{C2^k}{(1-\abs z)^n}\normw f^2, \qquad z\in \Omega_k, k\ge 0,
\end{equation}
where $C$ is a positive constant independent of $k$ and $z$.
\end{proposition}
\begin{proof}
Let $z\in \Omega_k$. Applying \cite[Corollary 4.5]{zhuball} to the function $g(z)=f(r_{k+2}z)$ at the point $\frac{z}{r_{k+2}}$, we obtain 
$$ \abs{f(z)}^2 \le \int_{\Sn}\abs{f(r_{k+2}\zeta)}^2\dfrac{(1-\abs{z/r_{k+2}}^2)^n}{\abs{1-\vh{z/r_{k+2},\zeta}}^{2n}}d\sigma(\zeta). $$
By \eqref{1}, $\abs{1-\vh{z/r_{k+2},\zeta}}\ge 1-\abs{\vh{z/r_{k+2},\zeta}}\ge 1-\frac{\abs z\abs \zeta}{r_{k+2}} = 1-\abs z/r_{k+2} \gess 1-\abs z$ for $z\in \Omega_k, \zeta \in \Sn$. Thus,
\begin{align*}
\abs{f(z)}^2&\less \int_{\Sn}\abs{f(r_{k+2}\zeta)}^2\dfrac{(1-\abs z^2)^n}{(1-\abs z)^{2n}}d\sigma(\zeta)\\
			& \le \dfrac{(1+\abs z)^n}{(1-\abs z)^n}\int_{\Sn}\abs{f(r_{k+2}\zeta)}^2d\sigma(\zeta)\\
			&\less \dfrac{2^k}{(1-\abs z)^n} 2^{-k}\int_{\Sn}\abs{f(r_{k+2}\zeta)}^2d\sigma(\zeta)\\
			&\le \dfrac{2^k}{(1-\abs z)^n}\sum_{j=1}^\infty 2^{-j}\int_{\Sn}\abs{f(r_{j+2}\zeta)}^2d\sigma(\zeta)\\
			&\lesssim \dfrac{2^k}{(1-\abs z)^n} \normw f^2,
\end{align*}
with constants independent of $k$ and $z$.
\end{proof}

\begin{corollary}
A sequence of functions $\set{f_k} \subset \A2w$ converges to $0$ weakly in $\A2w$ if and only if it is bounded in $\A2w$ and converges to $0$ uniformly on each compact subset of $\Bn$.
\end{corollary}

\begin{proof}[Proof of part $(ii)$ of Theorem \ref{Toeplitz}]
Suppose that $T_\mu$ is compact on $\A2w$. We define $h_a, a\in \Bn$ by \eqref{5} and $\widetilde{T_\mu}$ by \eqref{7}. Then $\normw{h_a}^2 \asymp 1$ and $h_a$ converges uniformly to $0$ on compact subsets of $\Bn$ as $\abs a\to 1$. Since $T_\mu$ is compact, $\widetilde{T_\mu}(a)\to 0$ as $\abs a\to 1$. By \eqref{8} this implies that
$$ \sup _{a\in \Omega_k}\dfrac{2^k\mu_k(Q_a)}{(1-\abs a)^n}\to 0  \text{ as } k\to \infty.$$
Hence, 
$$\lim _{k\to \infty}2^{k}\mathcal C_{\mu_k}(H^2)=0.$$
By part $(ii)$ of Theorem \ref{Carleson}, $\mu$ is a vanishing Carleson measure for $\A2w$.

Conversely, assume that $\mu$ is a vanishing Carleson measure for $\A2w$. For every $h\in \A2w$ we have 
$$\normw{T_\mu h}= \sup_{\substack{g\in \A2w \\ \normw{g} \le 1 }} \abs{\vhw{T_\mu h,g}}.$$
Furthermore, 
\begin{align*}
\abs{\vhw{T_\mu h,g}}&=\left|\int_{\Bn}h(z)\overline{g(z)}d\mu(z)\right|\le \int_{\Bn}\abs{h(z)}\abs{g(z)}d\mu(z) \\
			&\le \left (\int_{\Bn}\abs{h(z)}^2d\mu(z)\right )^{1/2} \left (\int_{\Bn}\abs{g(z)}^2d\mu(z)\right )^{1/2}\\
			&\lesssim \left (\int_{\Bn}\abs{h(z)}^2d\mu(z)\right )^{1/2} \normw{g}.
\end{align*}
The last inequality follows from the fact that $\mu$ is a Carleson measure for $\A2w$. Therefore,
$$\normw{T_\mu h}\lesssim \left (\int_{\Bn}\abs{h(z)}^2d\mu(z)\right )^{1/2}, \qquad h\in \A2w. $$
Now, let $\set{f_k}\subset \A2w$ be bounded and converge uniformly to $0$ on compact subsets of $\Bn$. Since $\mu$ is a vanishing Carleson measure for $\A2w$, 
$$\lim _{k\to \infty}\int_{\Bn}\abs{f_k(z)}^2d\mu(z)=0.$$ 
It follows that $\normw{T_\mu f_k}\to 0$ and hence $T_\mu$ is compact.
\end{proof}

\section{Proof of Theorem \ref{Schatten}}\label{sec5}
\begin{proposition} \label{md7}
Let $K_\rho(z,w)$ be the reproducing kernel of $\A2w$.
\begin{itemize}
\item[(a)] Let  $k\ge 1$, $z\in \Omega_k$. Then
\begin{equation}\label{10}
	K_\rho(z,z)\asymp \dfrac{2^k}{(1-\abs z)^n}.
\end{equation}
\item[(b)] There exists $\alpha=\alpha(\rho)>0$ such that for every $z\in \Bn$,
\begin{equation}\label{11}
\abs{K_\rho(z,w)}^2\asymp K_\rho(z,z) K_\rho(w,w)
\end{equation}
whenever $w\in E(z,\alpha)$.
\end{itemize}
\end{proposition}

\begin{proof}
$(a)$ Fix $k\ge 1$. Given $z\in \Omega_k$, let $L_z$ be the point evaluation at $z$ on $\A2w$. It is well known that 
$$K_\rho(z,z)=\norm{L_z}^2.$$
By Proposition \ref{md4}, $$\norm{L_z}^2\lesssim \dfrac{2^k}{(1-\abs z)^n}.$$ 
Furthermore, choosing $h_z$ by \eqref{5}, we have $\normw{h_z}\asymp 1$ and  $$\abs{h_z(z)}^2\gtrsim \dfrac{2^k}{(1-\abs z)^n}.$$ Hence,
$$\norm{L_z}^2\gtrsim \dfrac{2^k}{(1-\abs z)^n}.$$
Thus $$K_\rho(z,z)\asymp \dfrac{2^k}{(1-\abs z)^n},  \qquad z \in \Omega_k.$$

$(b)$ It is well known that 
$$\abs{K_\rho(z,w)}^2\le K_\rho(z,z) K_\rho(w,w)$$
for all $z, w\in \Bn.$
For any fixed $z_0\in \Omega_k$, consider the subspace $\A2w(z_0)$  defined as
$$\A2w(z_0)=\set{f\in \A2w: f(z_0)=0}.$$ 
Denote by $\L_{z_0} $ the one-dimensional subspace spanned by the function 
$$k_{\rho,z_0}(z)=\dfrac{K_\rho(z, z_0)}{\sqrt{K_\rho(z_0, z_0)}}.$$
Then we have the orthogonal decomposition
$$\A2w=\A2w(z_0) \oplus \L_{z_0}.$$
Hence $K_\rho(z,w)=K_{\rho,z_0}(z,w)+\overline{k_{\rho,z_0}(w)}k_{\rho,z_0}(z)$, where $K_{\rho,z_0}$ is the reproducing kernel of $\A2w(z_0)$.
Therefore,
$$K_\rho(z_0, w)=\overline{k_{\rho,z_0}(w)}k_{\rho,z_0}(z_0)$$
and
\begin{equation}\label{12}
K_\rho(w,w)=K_{\rho,z_0}(w,w)+\abs{k_{\rho,z_0}(w)}^2.
\end{equation}
We are going to prove that there exist $\alpha>0$ such that 
\begin{equation}\label{13}
K_{\rho,z_0}(w,w)<\dfrac12 K_\rho(w,w), \qquad w\in E(z_0,\alpha).
\end{equation}
By \eqref{1}, there exists $\alpha_1 >0$ such that $E(z_0,\alpha)\subset \Omega_{k-1}\cup \Omega_k\cup \Omega_{k+1}, 0<\alpha<\alpha_1$. Hence, for every $f\in \A2w(z_0)$ such that $\normw f =1$, by Proposition \ref{md4} we have 
\begin{equation}\label{14}
\abs{f(w)}^2\less \dfrac{2^k}{(1-\abs w)^n}\asymp \dfrac{2^k}{(1-\abs{z_0})^n}
\end{equation}
whenever $w\in E(z_0,\alpha)$. Since $E(z_0,\alpha)=\varphi_{z_0}\bigl(E(0,\alpha)\bigr)$, we can rewrite \eqref{14} as
\begin{equation}\label{15}
\abs{f\bigl(\varphi_{z_0}(\eta)\bigr)}^2\less \dfrac{2^k}{(1-\abs{z_0})^n}
\end{equation}
whenever $\eta \in E(0,\alpha)$.
Note that $f(z_0)=f\bigl(\varphi_{z_0}(0)\bigr)=0$. Therefore, by the Schwarz lemma, we get
\begin{equation*}
\abs{f\bigl(\varphi_{z_0}(\eta)\bigr)}^2\less \abs \eta ^2 \dfrac{2^k}{(1-\abs{z_0})^n} \asymp \abs \eta ^2 \dfrac{2^k}{(1-\abs{\varphi_{z_0}(\eta)})^n}
\end{equation*}
whenever $\eta \in E(0,\alpha)$.
This implies that there is a constant $C>0$ such that 
\begin{equation*}
\abs{f\bigl(\varphi_{z_0}(\eta)\bigr)}^2 \le C \abs \eta ^2 \dfrac{2^k}{(1-\abs{\varphi_{z_0}(\eta)})^n}, \qquad \eta \in E(0,\alpha).
\end{equation*}
Therefore, we can choose $\alpha$ so small that
\begin{equation*}
\abs{f\bigl(\varphi_{z_0}(\eta)\bigr)}^2 < \dfrac12 K_\rho\bigl(\varphi_{z_0}(\eta),\varphi_{z_0}(\eta)\bigr), \qquad \eta \in E(0,\alpha).
\end{equation*}
This proves \eqref{13}.

Now, from \eqref{12} and \eqref{13}, we obtain that $\abs{k_{\rho,z_0}(w)}^2 > \dfrac12 K_\rho(w,w)$ whenever $w\in E(z_0,\alpha)$. This means that
 $$\abs{K_\rho(w,z_0)}^2>\dfrac12 K_\rho(z_0,z_0) K_\rho(w,w)$$
 whenever $ w\in E(z_0,\alpha)$, which completes the proof.
 \end{proof}

\begin{lemma}\label{bd8}
Let $T$ be a positive operator on $\A2w$, and let $\widetilde T$ be the Berezin transform of $T$, defined by
$$\widetilde{T}(z)=\vhw{Tk_z, k_z}, \qquad z\in \Bn.$$
\begin{itemize}
\item[(a)] Let $0<p\le 1$. If $\widetilde{T}\in L^{p}(\Bn, d\lambdaw)$, then $T$ is in $\mathcal S_p.$
\item[(b)] Let $p\ge 1$. If $T$ is in $\mathcal S_p$, then $\widetilde T \in L^{p}(\Bn, d\lambdaw).$
\end{itemize}
Here, $d\lambdaw(z)=\dfrac{2^k \rho(z)dv(z)}{(1-\abs z)^n}$ if $z\in \Omega_k.$
\end{lemma}

\begin{proof}
Note that $d\lambdaw(z)\asymp K(z,z) \rho(z)dv(z)=\norm{K_z}^2\rho(z)dv(z)$. 

The proof is similar to the proof of  \cite[Lemma 4.2]{Arroussi2015}.
The positive operator $T$ is in $\mathcal S_p$ if and only if $T^p$ is in the trace class $\mathcal S_1$. Fix an orthonormal basis $\set{e_k}$ of $\A2w$. Since $T^p$ is positive, it is in $S_1$ if and only if $\sum_k\vhw{T^pe_k, e_k}<\infty.$ Let $U=\sqrt{T^p}$. By Fubini's theorem,  the reproducing property of $K_z$, and Parseval's identity, we have
\begin{align*}
\sum_k\vhw{T^pe_k, e_k}
			&=\sum_k\normw{Ue_k}^2=\sum_k\int_{\Bn}\abs{Ue_k(z)}^2\rho(z)dv(z)\\
			&=\int_{\Bn}\left (\sum_k\abs{Ue_k(z)}^2\right )\rho(z)dv(z)\\
			&=\int_{\Bn}\left (\sum_k\abs{\vhw{Ue_k,K_z}}^2\right )\rho(z)dv(z)\\
			&=\int_{\Bn}\left (\sum_k\abs{\vhw{e_k,UK_z}}^2\right )\rho(z)dv(z)\\
			&=\int_{\Bn}\normw{UK_z}^2\rho(z)dv(z)\\
			&=\int_{\Bn}\vhw{T^pK_z, K_z}\rho(z)dv(z)\\
			&=\int_{\Bn}\vhw{T^pk_z, k_z}\normw{K_z}^2\rho(z)dv(z)\\
			&\asymp \int_{\Bn}\vhw{T^pk_z, k_z}d\lambdaw(z).
\end{align*}
Hence, both (a) and (b) are the consequences of the well known inequalities (see \cite[Proposition 1.31]{zhudisk})
\begin{align*}
\vhw{T^pk_z, k_z}&\le \vhw{Tk_z, k_z}^p=\bigl(\widetilde T(z)\bigr)^p, \qquad  0<p\le 1,\\
\vhw{T^pk_z, k_z}&\ge \vhw{Tk_z, k_z}^p=\bigl(\widetilde T(z)\bigr)^p, \qquad p\ge 1. \qedhere
\end{align*}
\end{proof}

\begin{lemma}\label{bd9}
Let $\rho \in S^*$ and $z\in \Omega_k$. Then there exists $\alpha_0>0$ such that for every $\alpha\in (0,\alpha_0)$ we have
$$\abs{f(z)}^2\less \dfrac{2^k}{(1-\abs z)^n}\int_{E(z, \alpha)}\abs{f(w)}^2\rho(w)dv(w)$$
for all $f\in H(\Bn)$.
\end{lemma}

\begin{proof} 
 Let $z\in \Omega_k$. For each $f\in H(\Bn)$, by the subharmonicity of the function $w\mapsto \abs{f(w)}^2$ and the estimate $v\bigl(E(z,\alpha)\bigr)\asym (1-\abs z)^{n+1}$, we have
 \begin{equation*}
 \abs{f(z)}^2\lesssim \dfrac{1}{(1-\abs z)^{n+1}}\int_{E(z, \alpha)}\abs{f(w)}^2dv(w).
 \end{equation*}
 It is easy to see that $1-\abs z\asymp 1-\abs w$ for $w\in E(z, \alpha)$. Hence,
 \begin{align} \label{16}
 \abs{f(z)}^2
 			&\lesssim \dfrac{1}{(1-\abs z)^{n}}\int_{E(z, \alpha)}\abs{f(w)}^2\dfrac{1}{1-\abs w}dv(w) \notag\\
 			&=\dfrac{2^k}{(1-\abs z)^{n}}\int_{E(z, \alpha)}\abs{f(w)}^2\dfrac{2^{-k}}{1-\abs w}dv(w).
 \end{align}
 By \eqref{1}, for small $\alpha_0$ we have $E(z,\alpha_0)\subset \Omega_{k-1}\cup \Omega_k\cup \Omega_{k+1}$. Therefore, for every $\alpha\in (0,\alpha_0)$, we have $r_{k-1}<\abs w< r_{k+2}$ for $w\in E(z,\alpha)$.
Since $\int_{r_{k+2}}^1\rho(t)dt=2^{-k-2}$, we obtain
 $ 2^{-k} \lesssim \int_{\abs w}^1\rho(t)dt $
 for every $ w\in E(z,\alpha), \alpha\in (0,\alpha_0)$.
 Plugging this into \eqref{16} and using that $\rho^*(w)\less \rho(w)$, we get
 \begin{align*}
 \abs{f(z)}^2 
 			&\lesssim \dfrac{2^k}{(1-\abs z)^{n}}\int_{E(z, \alpha)}\abs{f(w)}^2\rho^*(w)dv(w)\\
 			&\lesssim \dfrac{2^k}{(1-\abs z)^{n}}\int_{E(z, \alpha)}\abs{f(w)}^2\rho(w)dv(w).
 \end{align*}
 This completes the proof.
\end{proof}

\begin{proof}[\textbf{Proof of Theorem \ref{Schatten}.}]
$(a) \Rightarrow (b)$. This follows from Lemma \ref{bd8} $(b)$.

$(b)\Rightarrow (c)$. By Proposition \ref{md7} $(b)$, for sufficiently small $\alpha>0$, we have
$$ \abs{K_z(w)}^{2}\asymp \normw{K_z}^2\normw{K_w}^2, \qquad w\in E(z,\alpha), z\in \Bn.$$
Then by Proposition \ref{md7} $(a)$, we get
\begin{align*}
\widetilde{T_\mu}(z)
			&=\int_{\Bn}\abs{k_z(w)}^2d\mu(w)=\normw{K_z}^{-2}\int_{\Bn}\abs{K_z(w)}^2d\mu(w)\\
			&\ge \normw{K_z}^{-2}\int_{E(z,\alpha)}\abs{K_z(w)}^2d\mu(w)\\
			&\asymp \int_{E(z,\alpha)}\normw{K_w}^{2}d\mu(w) \asymp \widehat{\mu}_\alpha(z).
\end{align*}
Since $\widetilde{T_\mu}$ is in $L^{p}(\Bn, d\lambdaw)$, $\widehat{\mu}_\alpha$ is also in $L^{p}(\Bn, d\lambdaw)$.

$(c)\Rightarrow (a)$. For every orthonormal basis $\set{e_l}$ of $\A2w$, we have
\begin{equation}\label{17}
\sum_{l}\vhw{T_\mu e_l, e_l}^{p}=\sum_{l}\left (\int_{\Bn}\abs{e_l(z)}^2d\mu(z)\right )^{p}.
\end{equation}
By Lemma \ref{bd9}, 
$$\abs{e_l(z)}^2\lesssim \dfrac{2^k}{(1-\abs z)^n}\int_{E(z,\alpha)}\abs{e_l(w)}^2\rho(w)dv(w),  \qquad z\in \Omega_k.$$
By Fubini's theorem and H\"older's inequality, we have
\begin{align*}
\int_{\Bn}\abs{e_l(z)}^2d\mu(z)
			&\lesssim \int_{\Bn}\abs{e_l(w)}^2 \widehat{\mu}_\alpha(w) \rho(w)dv(w)\\
			&\le \left (\int_{\Bn} \abs{e_l(w)}^{2}\widehat{\mu}_\alpha(w)^p\rho(w)dv(w)\right )^{1/p} \\
			& \qquad \qquad \times \left (\int_{\Bn} \abs{e_l(w)}^{2}\rho(w)dv(w)\right )^{1/q} \\
			&=\left (\int_{\Bn} \abs{e_l(w)}^{2}\widehat{\mu}_\alpha(w)^p \rho(w)dv(w)\right )^{1/p},
\end{align*}
where $\frac{1}{p}+\frac{1}{q}=1.$
Thus, \eqref{17} implies that 
\begin{align*}
\sum_{l}\vhw{T_\mu e_l, e_l}^{p}
		&\lesssim \int_{\Bn}\left ( \sum_l\abs{e_l(w)}^{2}\right )\widehat{\mu}_\alpha(w)^p\rho(w)dv(w)\\
		&=\int_{\Bn} \normw{K_w}^2\widehat{\mu}_\alpha(w)^p\rho(w)dv(w)\\
		&\asymp \int_{\Bn} \widehat{\mu}_\alpha(w)^p d\lambdaw(w) <\infty.
\end{align*}
This proves $(a)$.
\end{proof}

\begin{remark} Let $1<p<\infty$. In the case of large weighted Bergman spaces, Arroussi, Park and Pau proved in \cite[Theorem 4.6]{Arroussi2015} that
$$
T_{\mu} \in \mathcal S_p \Longleftrightarrow \widetilde{\mu}_{\varepsilon}(z)=\dfrac{\mu\bigl(B(z,\varepsilon)\bigr)}{(1-\abs z)^{2n}} \text{ is in the corresponding weighted } L^p,
$$
where $B(z,\varepsilon)$ is the Euclidean ball  with center $z$ and radius $\varepsilon(1-\abs{z})$. When the dimension $n=1$, we can see that $\widetilde{\mu}_{\varepsilon}$ is in $L^p$ if and only if $\widehat{\mu}_\varepsilon$ is in $L^p$. However, for $n>1$, this equivalence is not true anymore. 

Let us verify this. Choose $z_k\in \Bn$ such that $\abs{z_k}$ tend to $1$ sufficiently rapidly as $k\to \infty$. Consider 
$$
\mu=\sum_{k=1}^{\infty}c_k \chi_{B( z_k,\varepsilon)}\quad \text{ and }\quad \mu^*=\sum_{k=1}^{\infty}c_k \chi_{B( z_k,3\varepsilon)},
$$
where $c_k >0$ will be chosen later. We have 
$$
\mu\lesssim \widetilde{\mu}_{\varepsilon} \lesssim \mu^* 
$$
and 
$$
 \sum_{k=1}^\infty c_k \dfrac{v\bigl(B( z_k,\varepsilon)\bigr)}{v\bigl(E(z_k,\varepsilon)\bigr)} \chi_{E(z_k,\varepsilon)}\lesssim  \widehat{\mu}_\varepsilon \lesssim \sum_{k=1}^\infty c_k \dfrac{v\bigl(B( z_k,\varepsilon)\bigr)}{v\bigl(E(z_k,\varepsilon)\bigr)} \chi_{E(z_k,3\varepsilon)}.
$$ 
Hence 
\begin{align*}
\widetilde{\mu}_{\varepsilon} \in L^{p} &\Longleftrightarrow \sum_{k=1}^\infty c_k^p v\bigl(B(z_k,\varepsilon)\bigr)<\infty,\\
\intertext{and}
\widehat{\mu}_\varepsilon \in L^p &\Longleftrightarrow \sum_{k=1}^\infty c_k^p\dfrac{\bigl(v\bigl(B(z_k,\varepsilon)\bigr)\bigr)^p}{\bigl(v\bigl(E(z_k,\varepsilon)\bigr)\bigr)^{p-1}} <\infty.
\end{align*}
Since 
\begin{multline*}
\frac{c_k^p\bigl(v\bigl(B(z_k,\varepsilon)\bigr)\bigr)^p\bigl(v\bigl(E(z_k,\varepsilon)\bigr)\bigr)^{1-p}}{c_k^p v\bigl(B(z_k,\varepsilon)\bigr)}
=\left (\dfrac{v\bigl(B(z_k,\varepsilon)\bigr)}{v\bigl(E(z_k,\varepsilon)\bigr)}\right )^{p-1}\\ \asymp (1-\abs{z_k})^{(n-1)(p-1)}\longrightarrow 0 
\end{multline*} 
as $k\to \infty$, we can choose $c_k$ such that $\widehat{\mu}_\varepsilon \in L^p$ but $\widetilde{\mu}_{\varepsilon} \notin L^{p}$. On the other hand, one can easily see that $\widetilde{\mu}_{\varepsilon} \in L^{p}$ implies $\widehat{\mu}_\varepsilon \in L^p$.
\end{remark}

\vspace{0.2cm}
\noindent \textbf{Acknowledgments.} I am deeply grateful to my advisors Professors Alexander Borichev and El Hassan Youssfi for their help and many suggestions during the preparation of this paper.


\end{document}